\documentclass[11pt]{article}

\usepackage{amsmath,amsthm,amsfonts,amssymb, latexsym, eepic, floatflt, xypic, geometry,fancyhdr,cancel,color,picins}
\usepackage[dvips]{epsfig}
\newtheorem{thm}{Theorem}[section]
\newtheorem*{claim}{Claim}
\newtheorem{lem}[thm]{Lemma}

\theoremstyle{definition}
\newtheorem*{defin}{Definition}

\def\rr{\mathbb{R}}

\def\UU{\mathcal{U}}
\def\VV{\mathcal{V}}
\def\rmax{R_{\mathrm{max}}}

\def\rsch{R_{\mathrm{sc}}}
\def\rharm{R_{\mathrm{h}}}

\def\gflat{g_{\mathrm{flat}}}
\def\Gflat{G_{\mathrm{flat}}}

\def\isom{\cong}

\DeclareMathOperator{\spt}{spt}

\begin{document}

\title{On the Riemannian Penrose inequality in dimensions less than 8}
\author{Hubert L. Bray\thanks{The first author was partially supported by NSF grant DMS-0533551.}\\
Duke University\\
bray@math.duke.edu
 \and
Dan A. Lee\\ Duke University\\ dalee@math.duke.edu 
}
\date{\today}
\maketitle

\begin{abstract}
The Positive Mass Theorem states that a complete asymptotically flat manifold of nonnegative scalar 
curvature has nonnegative mass.  The Riemannian Penrose inequality provides a sharp lower bound for the 
mass 
when black holes are present.  More precisely, this lower bound is given in terms of the area of an 
outermost minimal surface, and equality is achieved only for Schwarzschild metrics.  The Riemannian 
Penrose inequality was first proved in three dimensions in 
1997 by G. Huisken and T. Ilmanen for the case of a single black hole \cite{huiilm}.  In 1999, H. Bray extended this result to the general case of multiple black holes using a different technique \cite{bray}.  In this paper we extend the technique of \cite{bray} to dimensions less than 8.
\end{abstract}

\section{Introduction}

The Penrose Conjecture is a longstanding conjecture in general relativity that provides a lower bound for the mass of an asymptotically flat spacelike slice of spacetime, in terms of the area of the black holes in the spacelike slice.  Penrose originally formulated the conjecture as a test for the far more ambitious idea of cosmic censorship.  In the case where the asymptotically flat spacelike slice is time-symmetric, the Penrose Conjecture reduces to a statement in Riemannian geometry, which we call the Riemannian Penrose inequality.  In this paper we will restrict our attention to the Riemannian Penrose inequality.  For more background on the general Penrose Conjecture, as well as some physical motivation, see \cite[Section 1]{bray} and references cited therein.

The Riemannian Penrose inequality was first proved in three dimensions in 1997 by G.~Huisken and T. Ilmanen for the case of a single black hole \cite{huiilm}.  In 1999, H. Bray extended this result to the general case of multiple black holes using a different technique \cite{bray}.  Before we state this theorem, let us review some definitions.

\begin{defin}
Let $n\geq 3$.  A Riemannian manifold $(M^n,g)$ is said to be \emph{asymptotically flat}\footnote{Note that there are various inequivalent definitions of asymptotic flatness in the literature, but they are all similar in spirit.  This one is taken from \cite[Section 4]{montecatini}.} if there is a compact set $K\subset M$ such that $M\smallsetminus K$ is a disjoint union of \emph{ends}, $E_k$, such that each end is diffeomorphic to $\rr^n\smallsetminus B_1(0)$, and in each of these coordinate charts, the metric $g_{ij}$ satisfies
\begin{eqnarray*}
g_{ij}&=&\delta_{ij}+O(|x|^{-p})\\
g_{ij,k}&=&O(|x|^{-p-1})\\
g_{ij,kl}&=&O(|x|^{-p-2})\\
R_g&=&O(|x|^{-q})
\end{eqnarray*}
for some $p>(n-2)/2$ and some $q>n$, where the commas denote partial derivatives in the coordinate chart, and $R_g$ is the scalar curvature of $g$.

In this case, in each end $E_k$, the limit
$$m(E_k,g)={1\over 2(n-1)\omega_{n-1}}\lim_{\sigma\to\infty}\int_{S_\sigma} (g_{ij,i}-g_{ii,j})\nu_jd\mu$$
exists (see, e.g.\ \cite[Section 4]{montecatini}), where  $\omega_{n-1}$ is the area of the standard unit $(n-1)$-sphere, $S_{\sigma}$ is the coordinate sphere in $E_k$ of radius $\sigma$, $\nu$ is its outward unit normal, and  $d\mu$ is the Euclidean area element on $S_{\sigma}$.  We call the quantity $m(E_k,g)$, first considered by Arnowitt, Deser, and Misner (see, e.g.\ \cite{ADM}), the \emph{ADM mass} of the end $(E_k,g)$, or when the context is clear, we simply call it the \emph{mass}, $m(g)$.  (Under an additional assumption on the Ricci curvature, R.\ Bartnik showed that the ADM mass is a Riemannian invariant, independent of choice of asymptotically flat coordinates \cite{bartnik}.)
\end{defin}

The Riemannian Penrose inequality may be thought of as a refinement of the celebrated Positive Mass Theorem when black holes are present.  Indeed, we will need to use the Positive Mass Theorem for our proof.
\begin{thm}[(Riemannian) Positive Mass Theorem]
Let $(M^n,g)$ be a complete asymptotically flat manifold with nonnegative scalar curvature.  If $n<8$ or if $M$ is spin, then the mass of each end is nonnegative.  Moreover, if any of the ends has zero mass, then $(M^n,g)$ is isometric to Euclidean space.
\end{thm}
The $n<8$ case was proved by R. Schoen and S.-T. Yau using minimal surface techniques~\cite{schyau} (see also \cite[Section 4]{montecatini}), and soon later E. Witten proved the spin case using a Bochner-type argument \cite{witten} (see also \cite{bartnik}).

Now fix a particular end of $M^n$.  Define $\mathcal{S}$ to be the collection of hypersurfaces that are smooth compact boundaries of open sets in $M$ containing all of the other ends.  Then each hypersurface in $\mathcal{S}$ defines a meaningful \emph{outside} and \emph{inside}.
\begin{defin}
A \emph{horizon} in $(M^n, g)$ is a minimal hypersurface in $\mathcal{S}$.  A horizon $\Sigma$ is \emph{outer minimizing} if its area minimizes area among all hypersurfaces in $\mathcal{S}$ enclosing $\Sigma$.
\end{defin}

\begin{defin}
The \emph{(Riemannian) Schwarzschild manifold} of dimension $n$ and mass $m$ is $\rr^n\smallsetminus\{0\}$ equipped with the metric $$g_{ij}(x)=\left(1+{m\over2}|x|^{2-n}\right)^{4\over n-2}\delta_{ij}.$$
Given a mass $m$, we also define the \emph{Schwarzschild radius} of the mass $m$ to be
$$\rsch(m)= \left({m\over 2}\right)^{1\over n-2}.$$
\end{defin}
Note that in a Schwarzschild manifold, the coordinate sphere of radius $\rsch(m)$ is the unique outer minimizing horizon, and its area $A$ satisfies the equation 
$$m= {1\over2} \left({A\over \omega_{n-1}}\right)^{n-2\over n-1}.$$

We can now state the main result of \cite{bray}.
\begin{thm}[Riemannian Penrose inequality in three dimensions]
Let $(M^3,g)$ be a complete asymptotically flat 3-manifold with nonnegative scalar curvature.  Fix one end.  Let $m$ be the mass of that end, and let $A$ be the area of an outer minimizing horizon (with one or more components).  Then 
$$m\geq\sqrt{{A\over 16\pi}},$$
with equality if and only if the part of $(M,g)$ outside the horizon is isometric to a Riemannian Schwarzschild manifold outside its unique outer minimizing horizon.
\end{thm}

Even though the original motivation from general relativity may have been specific to three dimensions, because of string theory there is a great deal of interest in higher dimensional black holes.  More importantly, from a purely geometric perspective, there appears to be nothing inherently three-dimensional about the Riemannian Penrose inequality, so it is natural to wonder whether the result holds in higher dimensions.  The goal of this paper is to prove the following generalization.
\begin{thm}[Riemannian Penrose inequality in dimensions less than 8]
Let $(M^n,g)$ be a complete asymptotically flat manifold with nonnegative scalar curvature, where $n<8$.  Fix one end. Let $m$ be the mass of that end, and let $A$ be the area of an outer minimizing horizon (with one or more components). Let $\omega_{n-1}$ be the area of the standard unit $(n-1)$-sphere.  Then
$$m\geq {1\over2} \left({A\over \omega_{n-1}}\right)^{n-2\over n-1},$$
with equality if and only if the part of $(M,g)$ outside the horizon is isometric to a Riemannian Schwarzschild manifold outside its unique outer minimizing horizon.
\end{thm}
Our proof is limited to dimensions less than 8 for the same reason that Schoen and Yau's proof is limited; we need to use regularity of minimal hypersurfaces.

The geometry inside the horizon plays no role at all in the proof of the theorem.  Accordingly, our objective is to prove the following theorem.
\begin{thm}\label{main}
Let $(M^n,g)$ be a complete one-ended asymptotically flat manifold with boundary, where $n<8$. If $(M,g)$ has nonnegative scalar curvature, and if the boundary is an outer minimizing horizon (with one or more components) with total area $A$, then
$$m\geq {1\over2} \left({A\over \omega_{n-1}}\right)^{n-2\over n-1}$$
with equality if and only if $(M,g)$ is isometric to a Riemannian Schwarzschild manifold outside its unique outer minimizing horizon.
\end{thm}
We will prove this theorem using the first author's conformal flow method \cite{bray}.  One might wonder whether Huisken and Ilmanen's inverse mean curvature flow method \cite{huiilm} could also be used for this purpose.  Unfortunately, since the Gauss-Bonnet Theorem lies at the heart of that method, it would require major new insights to adapt it higher dimensions.

\section{Overview of proof}
The vast majority of the first author's proof of the Riemannian Penrose inequality in dimension three applies to dimensions less than 8 \cite{bray}.  In this section we will review the main features of the proof and describe the parts that require modification.  Since a large portion of our proof is actually contained in \cite{bray}, we will try to maintain consistent notation.  The main technical tool that we will employ is the conformal flow.

\begin{defin}
Let $M$ be a manifold with a distinguished end.  Let $g_t$ be a family of metrics on $M$, and let $\Sigma(t)$ be a family of hypersurfaces in $\mathcal{S}$ such that $g_t(x)$ is Lipschitz in $t$, $C^1$ in $x$, and smooth in $x$ outside $\Sigma(t)$.  We say that $(M,g_t, \Sigma(t))$ is a \emph{conformal flow} if and only if the following conditions hold for each $t$:
\begin{itemize}
\item  $(M,g_t)$ outside $\Sigma(t)$ is a complete asymptotically flat manifold with boundary, and it has nonnegative scalar curvature.
\item  $\Sigma(t)$ is an outer minimizing horizon in $(M,g_t)$.
\item  ${d\over dt}g_t={4\over n-2}\nu_t g_t$, where $\nu_t(x)=0$ inside $\Sigma(t)$, and outside $\Sigma(t)$, $\nu_t$ is the unique solution to the Dirichlet problem
$$\left\{
\begin{array}{rcll}
 \Delta_{g_t}\nu_t(x)&=&0&\text{ outside }\Sigma(t)\\
\nu_t(x)&=&0&\text{ at }\Sigma(t)\\
\lim_{x\to\infty} \nu_t(x)&=&-1& 
\end{array}
\right.$$

\end{itemize}
\end{defin}
The formulation of the conformal flow in \cite{bray} is slightly different but defines the same flow.  Instead of using the last item in the above definition, we could set $g_t=u_t^{4\over n-2}g_0$ and demand that
$${d\over dt}u_t=v_t$$
where $v_t(x)=0$ inside $\Sigma(t)$, and outside $\Sigma(t)$, $v_t$ is the unique solution to the Dirichlet problem
$$\left\{
\begin{array}{rcll}
 \Delta_{g_0}v_t(x)&=&0&\text{ outside }\Sigma(t)\\
v_t(x)&=&0&\text{ at }\Sigma(t)\\
\lim_{x\to\infty} v_t(x)&=&-e^{-t}& 
\end{array}
\right.$$
The fact that these two formulations are equivalent follows from the following simple lemma, which we will use repeatedly.
\begin{lem}\label{simpleformula}
If $g_1$ and $g_2$ are smooth metrics and $\phi$ is a smooth function such that
$$g_2=\phi^{4\over n-2}g_1,$$
then for any smooth function $f$,
$$\Delta_{g_1}(f\phi)=\phi^{n+2\over n-2}\Delta_{g_2}f+f\Delta_{g_1}\phi.$$
\end{lem}

\begin{thm}
Given initial data $(M^n,g_0,\Sigma(0))$ satisfying the first two properties of the conformal flow described above, with $n<8$, there exists a conformal flow $(M,g_t,\Sigma(t))$ for all $t\geq0$.  Moreover,
\begin{itemize}
\item For all $t_2>t_1\geq0$, $\Sigma(t_2)$ encloses $\Sigma(t_1)$ without touching it.
\item $\Sigma(t)$ can ``jump'' at most countably many times.  At these jump times, we write $\Sigma^-(t)$ and $\Sigma^+(t)$ to denote the hypersurface ``before'' and ``after'' it jumps, respectively.\footnote{See \cite[Section 4]{bray} for a precise statement.  With the definition of the conformal flow given above, at a jump time, $\Sigma(t)$ could lie somewhere between $\Sigma^-(t)$ and $\Sigma^+(t)$.  However, in the construction of the conformal flow in \cite[Section 4]{bray}, $\Sigma(t)$ is the outermost horizon in $(M,g_t)$ containing $\Sigma(0)$, and consequently, we have $\Sigma(t)=\Sigma^+(t)$ for $t>0$.}
\end{itemize}
\end{thm}
The proof of this theorem in \cite[Theorem 2]{bray} is unchanged in higher dimensions, as long as $n<8$.  The basic idea behind the proof is to use a discrete time approximation, and then take the limit as the length of the discrete time intervals approaches zero.   The $n<8$ hypothesis is required in the proof in order to find smooth outermost minimal area enclosures.  The only other place that this dimensional restriction will be used again is when we invoke the Positive Mass Theorem.

To prove our main theorem (Theorem \ref{main}), we will use the hypotheses of the theorem as initial data for the conformal flow and 
prove that the conformal flow has the following properties:
\begin{itemize}
\item The area of $\Sigma(t)$ in $(M,g_t)$ is constant in $t$.  Call it $A$.
\item The mass of $(M,g_t)$, which we will call $m(t)$, is nonincreasing.
\item With the right choice of end coordinates, the metric $g_t$ outside $\Sigma(t)$ converges to a Schwarzschild metric.
\item The area of the horizon in this Schwarzschild manifold is greater than or equal to $A$.
\end{itemize}
Once we have established these properties, the main theorem follows immediately.\footnote{Except for the case of equality, which requires an additional simple argument.}

\begin{lem}
The area of $\Sigma(t)$ in $(M,g_t)$ is constant in $t$.
\end{lem}
The proof of this lemma in \cite[Section 5]{bray} is unchanged in higher dimensions.  The basic idea behind the proof is that the rate of change of the area $|\Sigma(t)|_{g_t}$ has a contribution from changing $\Sigma(t)$ while leaving $g_t$ fixed and a contribution from changing $g_t$ while leaving $\Sigma(t)$ fixed.  The first contribution is zero because $\Sigma(t)$ is minimal, and the second contribution is zero because the metric is not changing at $\Sigma(t)$.  (Specifically, ${d\over dt}g_t={4\over n-2}\nu_t g_t=0$ at $\Sigma(t)$.)  However, the proof is more subtle than this because $\Sigma(t)$ can jump.  See \cite[Section 5]{bray} for details.

In order to simplify the rest of our arguments, we use a tool called harmonic flatness.
\begin{defin}
A Riemannian manifold $(M^n,g)$ is said to be \emph{harmonically flat at infinity} if there is a compact set $K\subset M$ such that $M\smallsetminus K$ is the disjoint union of \emph{ends}, $E_k$, such that each end is diffeomorphic to $\rr^n\smallsetminus B_{r_k}(0)$, and in each of these coordinate charts, there is a (Euclidean) harmonic function $\UU$ such that
$$g_{ij}(x)=\UU(x)^{4\over n-2}\delta_{ij}.$$
In other words, each end is conformally flat with a harmonic conformal factor.
\end{defin}
Note that a harmonically flat end necessarily has zero scalar curvature.  Expanding $\UU$ in spherical harmonics in a particular end $E_k$, we see that
$$\UU(x)=a+b |x|^{2-n}+O(|x|^{1-n})$$
for some constants $a$ and $b$.  Clearly, a manifold that is harmonically flat at infinity is asymptotically flat.\footnote{However, when $a\neq 1$, it is necessary to change the distinguished coordinate chart.} A simple computation shows the following:
\begin{lem}
In the situation described above, the mass of the end $E_k$ is equal to $2ab$.
\end{lem}
A short argument of Schoen and Yau \cite{schyau} (see also \cite[Section 4]{montecatini}, \cite[Section 2]{bray}) implies the following lemma.
\begin{lem}
In order to prove our main theorem (Theorem \ref{main}), we may assume without loss of generality that $(M^n,g)$ is harmonically flat at infinity.
\end{lem}
From now on we will always work in the situation of initial data that is harmonically flat at infinity, and then evolved by the conformal flow.  From Lemma \ref{simpleformula} it is clear that the conformal flow preserves harmonic flatness outside $\Sigma(t)$.

We now consider a third formulation of the conformal flow.  By the harmonic flatness assumption, we know that $(g_0)_{ij}(x)=\UU_0(x)^{4\over n-2} \delta_{ij}$ for some harmonic function $\UU_0$ on the exterior region $\rr^n\smallsetminus B_{\rharm}$ for some $\rharm$.  (We adopt the shorthand notation $B_{R}=B_{R}(0)$ and $S_{R}=S_{R}(0)$.)  We choose end coordinates so that $\lim_{x\to\infty}\UU_0(x)=1$, and consequently, we are not allowed to choose the constant $\rharm$ arbitrarily.  Now extend $\UU_0$ to a positive function on all of $M$ and define the metric $\gflat$ by
$$g_0=\UU_0^{4\over n-2} \gflat.$$
Note that for $|x|>\rharm$, $(\gflat)_{ij}(x)=\delta_{ij}$.  We can now reformulate the conformal flow by setting
$$g_t=\UU_t^{4\over n-2} \gflat$$
and demanding that
$${d\over dt}\UU_t=\VV_t$$
where $\VV_t(x)=0$ inside $\Sigma(t)$, and outside $\Sigma(t)$, $\VV_t$ is the unique solution to the Dirichlet problem\footnote{We know that there is a unique solution because this formulation is equivalent to the previous ones.}
$$\left\{
\begin{array}{rcll}
 \Delta_{\gflat}\VV_t-\left({\Delta_{\gflat}\UU_0\over\UU_0}\right)\VV_t &=&0&\text{ outside }\Sigma(t)\\
\VV_t(x)&=&0&\text{ at }\Sigma(t)\\
\lim_{x\to\infty} \VV_t(x)&=&-e^{-t}& 
\end{array}
\right.$$
Note that in the region outside $\Sigma(t)$ with $|x|>\rharm$, both $\UU_t$ and $\VV_t$ are (Euclidean) harmonic functions.  We summarize the relationships between the three formulations of the conformal flow:
$$\UU_t=u_t\UU_0$$
$$\VV_t=v_t\UU_0$$
$$\nu_t={v_t\over u_t}={\VV_t\over \UU_t}.$$

\begin{lem}\label{monotone}
The mass, $m(t)$, is nonincreasing.
\end{lem}
The proof of this lemma in \cite[Sections 6 and 7]{bray} is also unchanged in higher dimensions.  However, the main idea used in the proof is central to this paper, so we will summarize the basic argument.

For each time $t$, consider the two-ended manifold $(\bar{M}_{\Sigma(t)},\bar{g}_t)$ obtained by reflecting the manifold $(M,g_t)$ through $\Sigma(t)$.  Let $\omega_t$ be the $g_t$-harmonic function that approaches $1$ at one end and $0$ at the other end.  We can use $\omega_t$ to conformally close the $0$-end by considering the metric $\tilde{g}_t=(\omega_t)^{4\over n-2}\bar{g}_t$ on $\bar{M}_{\Sigma(t)}$.  The result is a new one-ended manifold $(\tilde{M}_{\Sigma(t)}=\bar{M}_{\Sigma(t)}\cup\{\mathrm{pt}\},\tilde{g}_t)$ with nonnegative scalar curvature.\footnote{One can show that the singularity at $\mathrm{pt}$ is removeable.}  Similarly, if $t$ is a jump time, then we can construct $(\tilde{M}_{\Sigma^\pm(t)},\tilde{g}^\pm_t)$ by first reflecting through $\Sigma^\pm(t)$.  Lemma \ref{monotone} will follow from the following key lemma.

\begin{lem}\label{massderivative}
Let $\tilde{m}(t)$ be the mass of $(\tilde{M}_{\Sigma(t)},\tilde{g}_t)$.  If $t$ is not a jump time, then
$${d\over dt}m(t)=-2\tilde{m}(t).$$
If $t$ is a jump time, let $\tilde{m}^{\pm}(t)$ be the mass of $(\tilde{M}_{\Sigma^{\pm}(t)},\tilde{g}^\pm_t)$.  Then
$${d\over dt^\pm}m(t)=-2\tilde{m}^\pm(t),$$
where ${d\over dt^\pm}m(t)$ denotes the right and left side limits of ${d\over dt}m(t)$.

\end{lem}
Lemma \ref{monotone} follows from this lemma because the Positive Mass Theorem tells us that $\tilde{m}(t)\geq0$ (and that $\tilde{m}^\pm(t)\geq0$).  However, there is a technical point to deal with here:  Since the metric $\tilde{g}_t$ is not smooth along $\Sigma(t)$ where the gluing took place, the standard version of the Positive Mass Theorem does not immediately apply.  However, since $(\bar{M}_{\Sigma(t)},\bar{g}_t)$ was obtained by reflecting through a \emph{minimal} surface, one can show that $(\tilde{M}_{\Sigma(t)},\tilde{g}_t)$ is a limit of smooth manifolds with nonnegative scalar curvature, and we still have the desired result.  This argument was carried out in \cite[Section 6]{bray} and described in further detail in \cite{miao}.

\begin{proof}
For ease of notation, let us assume that $t$ is not a jump time.  (The proof for jump times is the same, but with $\pm$ superscripts everywhere.)
By symmetry, we know that the function $\omega_t$ used in the construction of $\tilde{g}_t$ must be ${1\over 2}(1-\nu_t)$ on one end (and ${1\over 2}(1+\nu_t)$ on the end to be closed up).  Therefore, in the one end of $\tilde{M}_{\Sigma(t)}$, for $|x|>\rharm$,
\begin{eqnarray*}
(\tilde{g}_t)_{ij}(x)&=&\left[{1\over 2}(1-\nu_t(x))\right]^{4\over n-2} (g_t)_{ij}(x)\\ 
&=&\left[{1\over 2}(1-\nu_t(x))\UU_t(x)\right]^{4\over n-2} \delta_{ij}\\ 
&=& \left[{1\over 2}(\UU_t(x)-\VV_t(x))\right]^{4\over n-2} \delta_{ij}
\end{eqnarray*}
We will now compute $\tilde{m}(t)$ by expanding ${1\over 2}(\UU_t(x)-\VV_t(x))$.  For $|x|>\rharm$, $\UU_t(x)$ is harmonic and thus we can expand it as
$$\UU_t(x)=A(t)+B(t)|x|^{2-n}+O(|x|^{1-n}).$$
Therefore
$$\VV_t(x)=A'(t)+B'(t)|x|^{2-n}+O(|x|^{1-n}).$$
We know that $A(0)=1$ and $A'(t)=\lim_{x\to\infty} \VV_t(x)=-e^{-t}$, so we can write
$$\UU_t(x)=e^{-t}+{1\over 2}e^t m(t)|x|^{2-n}+O(|x|^{1-n})$$
$$\VV_t(x)=-e^{-t}+{1\over 2}e^t (m(t)+m'(t))|x|^{2-n}+O(|x|^{1-n})$$
$${1\over 2}(\UU_t(x)-\VV_t(x))=e^{-t}-{1\over4}e^t m'(t)  |x|^{2-n}+O(|x|^{1-n})    $$
Thus $\tilde{m}(t)=-{1\over 2}m'(t)$.

\end{proof}

Therefore, in order to prove Theorem \ref{main}, the only part of \cite{bray} that needs to be modified is the part that deals with the convergence to Schwarzschild.  The basic idea here is that since $m(t)$ is nonincreasing and bounded below by zero (by Positive Mass Theorem), we might hope that its derivative, $-2\tilde{m}(t)$, converges to zero.  Indeed, that turns out to be the case (see Lemma \ref{mtildelimit}).  The equality case of the Positive Mass Theorem states that the only complete asymptotically flat manifold of nonnegative scalar curvature and zero mass is Euclidean space.  Therefore we might also hope that since $\tilde{m}(t)$ is converging to zero, $\tilde{g}_t$ must converge to the flat metric at infinity, in some sense.  In order to establish this fact, we will need a strengthened version of the equality case of the Positive Mass Theorem (see Theorem \ref{strongPMT}), proved in a separate paper \cite{strongPMT}.   Then it is not hard to see that (with the right choice of end coordinates), $g_t$ must converge to a Schwarzschild metric outside $\Sigma(t)$.\footnote{This is a refined version of the fact that the only asymptotically flat manifolds that are scalar-flat and conformal to Euclidean $\rr^{n}\smallsetminus\{0\}$ are Schwarzschild manifolds.}  

In order to make this basic argument work, we need to control $\Sigma(t)$.  (Specifically, we need Lemma \ref{rmax}.)  In \cite{bray}, this control was obtained using curvature estimates by way of the Gauss-Bonnet Theorem, together with a Harnack-type inequality from \cite{brayiga} that was only applicable in three dimensions.  It is this part of the proof that needs to be completely re-worked for application to higher dimensions.  Even though our new proof is more general, it is actually more elementary and straightforward than the original proof.  This content appears in Section \ref{rmaxproof} of this paper.

Section \ref{schwarzschild} of this paper serves as a replacement of Sections 8 through 12 in \cite{bray}, although there is a fair amount of overlap.  To summarize the differences: First, using the three-dimensional curvature estimates described above, it was proved in \cite{bray} that $\Sigma(t)$ eventually encloses any bounded region, and consequently one can then assume that $M$ is an exterior domain of $\rr^3$.  It turns out that this simplification is not actually needed for our proof, but it means that we have to be a bit more careful than in \cite{bray}.  Second, the strengthened version of the equality case of the Positive Mass Theorem (Theorem \ref{strongPMT}) mentioned above was proved in \cite{bray} using spinors.  We need a different proof here since higher dimensional manifolds need not be spin; the proof is given in \cite{strongPMT}.  Third, with the benefit of hindsight we are able to simplify and streamline many aspects of the original proof.

\section{Convergence to Schwarzschild}\label{schwarzschild}

As mentioned earlier, we want to show that $\tilde{m}(t)$ converges to zero as $t\to\infty$.
\begin{lem}\label{mtildelimit}
$$\lim_{t\to\infty} \tilde{m}(t)=0.$$
\end{lem}

\begin{proof}
\begin{claim}
The quantity $e^{2t}(m(t)+m'(t))$ is nondecreasing in $t$.
\end{claim}
Recall that for $t_2>t_1$, $\Sigma(t_2)$ encloses $\Sigma(t_1)$.  By the maximum principle and the definition of $v_t$, it is evident that $e^{t}v_t(x)$ is nondecreasing in $t$, for any fixed $x$.  Therefore $e^t\VV_t(x)$ is also nondecreasing in $t$, for any fixed $x$.  Recall from the proof of Lemma \ref{massderivative} that
\begin{equation*}
e^t\VV_t(x)=-1 +{1\over 2}e^{2t} (m(t)+m'(t))|x|^{2-n}+O(|x|^{1-n}).
\end{equation*}
The claim follows.

For now assume that $m(t)$ is smooth.
\begin{claim}
$$\tilde{m}'(t)\leq m(0).$$
\end{claim}
Differentiating the monotone quantity from the previous claim,
$$0\leq {d\over dt}\left[e^{2t}(m(t)+m'(t))\right]=e^{2t}(m''(t)+3m'(t)+2m(t)).$$
Since $m'(t)=-2\tilde{m}(t)\leq 0$, we have 
$$0\leq m''(t)+2m(t)\leq -2\tilde{m}'(t)+2m(0),$$
proving the claim.  Since $\tilde{m}(t)$ is a nonnegative function with finite integral and derivative bounded above, it follows that $\lim_{t\to\infty} \tilde{m}(t)=0$.  Of course, $m(t)$ is not necessarily smooth, but it is a simple exercise to show that the result still holds.
\end{proof}

Since $m(t)$ is nonincreasing and bounded below by zero, it must have a limit.
\begin{lem}\label{Mnotzero}
Let $M=\lim_{t\to\infty}m(t)$.  $M>0$.
\end{lem}
We postpone the proof of this lemma until the next section, so as not to interrupt the flow of the main argument.

Let $r_0<{1\over2}\rsch(M)$, and choose a diffeomorphism $\rr^n\smallsetminus B_{r_0}\isom M\smallsetminus K$.  That is, we choose coordinates in $\rr^n\smallsetminus B_{r_0}$ for the end.  Recall that since we chose the normalization $\lim_{x\to\infty}\UU_0=1$, we \emph{cannot} say that $\UU_0$ is harmonic in $\rr^n\smallsetminus B_{r_0}$ without losing generality.  We can only say that $\UU_0$ is harmonic in $\rr^n\smallsetminus B_{\rharm}$ for some possibly large $\rharm$.  

We want to talk about convergence of our Riemannian manifold as $t\to\infty$, but we will see that, with respect to a \emph{fixed} coordinate system at infinity, $\Sigma(t)$ runs off to infinity.  Consequently, the part of $g_t$ that we care about (the part outside $\Sigma(t)$) disappears in the limit.  Therefore we need to change our choice of coordinates as $t$ changes.  One way to do this is to introduce a one-parameter group of diffeomorphisms.  Choose a smooth vector field $X$ on $M$ such that 
$$X={2\over n-2}r{\partial\over\partial r}$$
on $\rr^n\smallsetminus B_{r_0}$, where $r=|x|$ is the radial coordinate on $\rr^n\smallsetminus B_{r_0}$.  (We extend $X$ inside $K$ so that it is smooth.)
Let $\Phi_t$ be the one-parameter group of diffeomorphisms of $M$ generated by $X$.
\begin{defin}
Given our conformal flow $(M,g_t,\Sigma(t))$, we define the \emph{normalized conformal flow} $(M,G_t,\Sigma^*(t))$ by
$$G_t=\Phi_t^*g_t$$
$$\Sigma^*(t)=\Phi_t^{-1}(\Sigma(t)).$$
\end{defin}
Define new functions 
$$U_t(x)=e^t \UU_t(\Phi_t(x))$$
$$V_t(x)=e^t \VV_t(\Phi_t(x)),$$
and a new metric
$$(\Gflat)_t=e^{-4t\over n-2}\Phi_t^*\gflat.$$
Note that $G_t=U_t^{4\over n-2}(\Gflat)_t$.  Also note that $V_t(x)=0$ inside $\Sigma^*(t)$, and outside $\Sigma^*(t)$, $V_t$ is the unique solution to the Dirichlet problem
$$\left\{
\begin{array}{rcll}
 \Delta_{(\Gflat)_t}V_t-\left({\Delta_{(\Gflat)_t} U_0\over U_0}\right) V_t  &=&0\text{ outside }\Sigma^*(t)\\
V_t(x)&=&0\text{ at }\Sigma^*(t)\\
\lim_{x\to\infty} V_t(x)&=&-1& 
\end{array}
\right.$$
Differentiating the definition of $U_t$, we see that
$${d\over dt}U_t=V_t+U_t+XU_t.$$
For all $t>t_0={n-2\over 4}\log\left({\rharm\over r_0}\right)$ and $|x|>r_0$, we see that $((\Gflat)_t)_{ij}(x)=\delta_{ij}$ and
\begin{equation}\label{diffyqV}
{d\over dt}U_t=V_t + U_t + {2\over n-2}r{\partial\over\partial r}U_t.
\end{equation}
Since we are concerned with what happens for large $t$, from now on we will always assume that $t>t_0$.

Let $W_t={1\over 2}(U_t-V_t)$ outside of $\Sigma^*(t)$.  Now define
$$\tilde{G}_t=W_t^{4\over n-2}(\Gflat)_t$$
on the exterior of $\Sigma^*(t)$.  Observe that $\tilde{G}_t$ is isometric to $\tilde{g}_t$, and consequently it
has mass equal to $\tilde{m}(t)$. Note that for $t>t_0$ and $x$ outside $\Sigma^*(t)$ with $|x|>r_0$,
\begin{equation}\label{diffyqW}
{d\over dt}U_t=2\left[U_t-W_t + {1\over n-2}r{\partial\over\partial r}U_t\right].
\end{equation}

As mentioned earlier, in order to make this argument work, we need to obtain control on $\Sigma(t)$.  We postpone the proof of this lemma so as not to interrupt the flow of the main argument.
\begin{lem}\label{rmax}
There exists some $\rmax>0$ such that $\Sigma^*(t)$ is always enclosed by the coordinate sphere of radius $\rmax$.
\end{lem}

As mentioned earlier, we will need to use a strengthened version of the equality case of the Positive Mass Theorem.  Essentially, we want to say that a sequence of asymptotically flat manifolds of nonnegative scalar curvature becomes flatter as the mass approaches zero.  The proof of this theorem is the subject of a separate paper \cite{strongPMT}.

\begin{thm}\label{strongPMT}
Let $M^n$ be any smooth manifold on which the Positive Mass Theorem holds.  Let $\rr^n\smallsetminus B_{R}(0)$ be a coordinate chart for one of the ends of $M$.  Let $G$ be a complete asymptotically flat metric of nonnegative scalar curvature on $M$, and suppose that 
$$G_{ij}(x)=W(x)^{4\over n-2}\delta_{ij}$$
in  $\rr^n\smallsetminus B_{R}(0)$, where
$W$ is a positive (Euclidean) harmonic function on $\rr^n\smallsetminus B_{R}(0)$ with $\lim_{x\to\infty} W(x)=1$.  

Then for any $\epsilon>0$, there exists $\delta>0$ such that if the mass of $(M,G)$ in this end is less then $\delta R^{n-2}$, then 
$$\sup_{|x|>aR}|W(x)-1|<\epsilon,$$ 
where $a$ is a universal constant depending only on $n$.  The constant $\delta$ depends only on $\epsilon$ and $n$.  In particular, it does not depend on the topology of $M$.
\end{thm}

Now observe that $(\tilde{G}_t)_{ij}(x)=W_t(x)^{4\over n-2}\delta_{ij}$ in $\rr^n\smallsetminus B_{\rmax}$, $W_t$ is harmonic on $\rr^n\smallsetminus B_{\rmax}$, and $\lim_{t\to\infty}\tilde{m}(t)=0$.  Also observe that since $\tilde{G}_t$ is isometric to $\tilde{g}_t$, it is a limit of metrics that extend to complete metrics of nonnegative scalar curvature \cite{miao}.  In short, we may apply the previous theorem to $\tilde{G}_t$ in order to conclude that $\lim_{t\to\infty} W_t(x)=1$ uniformly for $|x|>a\rmax$.

\begin{lem}\label{converge}
The following limits hold uniformly over all $|x|\geq 2a\rmax$.
\begin{eqnarray*}
\lim_{t\to\infty} U_t(x)&=&1+{M\over2}|x|^{2-n}\\
\lim_{t\to\infty} V_t(x)&=&-1+{M\over2}|x|^{2-n}\\
\lim_{t\to\infty} W_t(x)&=&1.
\end{eqnarray*}
\end{lem}
\begin{proof}
Let
$$\bar{U}_t(x)=U_t(x)-\left(1+{m(t)\over2}|x|^{2-n}\right)$$
and
$$\bar{W}_t(x)=W_t(x)-\left(1+{\tilde{m}(t)\over2}|x|^{2-n}\right).$$
Therefore, by equation (\ref{diffyqW}) and Lemma \ref{massderivative}, we know that for $x\in\rr^n\smallsetminus B_{r_0}$ and outside $\Sigma^*(t)$,
\begin{equation}\label{barredeqn}
{d\over dt}\bar{U}_t=2\left[\bar{U}_t-\bar{W}_t + {1\over n-2}r{\partial\over\partial r}\bar{U}_t\right].
\end{equation}

Let $\epsilon>0$.  By the previous theorem and the following discussion, the third equality in the 
statement of the lemma follows immediately.  In other words, we know that for large enough $t$, 
$\tilde{m}(t)$ is small enough so that
$$\sup_{x\in S_{a\rmax}} |W_t(x)-1|<\epsilon.$$
So
\begin{equation*}
\sup_{x\in S_{a\rmax}} |\bar{W}_t(x)|  <  \epsilon+ {\tilde{m}(t)\over 2}(a\rmax)^{2-n}
< 2\epsilon
\end{equation*}
for large enough $t$.  
Since $\bar{W}_t$ is harmonic ans has no constant or $|x|^{2-n}$ terms in its expansion, it follows from 
the maximum principle 
and a gradient estimate that for all $|x|>2a\rmax$,
$$ |\bar{W}_t(x)| < C\epsilon|x|^{1-n},$$
for some constant $C$ independent of $\epsilon$.  Analyzing equation (\ref{barredeqn}), we conclude that 
for large enough $t$,
$$ |\bar{U}_t(x)| < 3C\epsilon|x|^{1-n}.$$
The first equation in the statement of the lemma now follows from the definition of  $\bar{U}_t$ and the 
fact that \hbox{$\lim_{t\to\infty}m(t)=M$}. (The second equation in the statement of the lemma follows 
from the other two.)
\end{proof}

\begin{lem}\label{lastlemma}
For $X\subset \rr^n$, $\epsilon>0$, let $(X)_\epsilon$ denote the $\epsilon$-neighborhood of $X$, that is, the set of points that are distance less than $\epsilon$ away from $X$.  For all $\epsilon>0$, there exists some large~$t$ such that 
$$\Sigma^*(t)\subset (S_{\rsch(M)})_\epsilon,$$
where $S_{\rsch(M)}$ is the sphere of radius $\rsch(M)=\left({M\over 2}\right)^{1\over n-2}$ in $\rr^n$.\footnote{Actually, it is only necessary to show that $\Sigma^*(t)$ lies within the sphere of radius $\rsch(M)+\epsilon$.}
\end{lem}

\begin{proof}
Using maximum principle arguments and Lemma \ref{rmax}, one can prove uniform upper and lower bounds on $U_t(x)$ on $\rr^n\smallsetminus B_{r_0}$.

Since the area of $\Sigma^*(t)$ with respect to $G_t$ is constant, and since there is a uniform lower 
bound on $U_t(x)$, we have a uniform upper bound on the Euclidean area of $\Sigma^*(t)\cap 
(\rr^n\smallsetminus B_{r_0})$.  Therefore we can show that for some sequence $t_i$, the part of 
$\Sigma^*(t_i)$ in $\rr^n\smallsetminus B_{r_0}$ weakly converges to some $\Sigma_\infty$.
But moreover, using the uniform bounds on $U_t(x)$, one can argue the stronger statement that 
$\Sigma^*(t_i)$ converges to $\Sigma_\infty$ in Hausdorff distance. (See \cite[Section 12 and Appendix 
E]{bray} for details.  One can also argue directly using Lemma \ref{gamma}.)

Since the $V_t$'s are harmonic and uniformly bounded, we can choose a subsequence such that $V_{t_i}$ converges uniformly on compact subsets of the exterior of $\Sigma_\infty$ in $\rr^n\smallsetminus B_{r_0}$.  Since the limit must be a harmonic function, it follows from the previous lemma that the limit is $V_\infty(x)=-1+{M\over2}|x|^{2-n}$.  More precisely, given $\epsilon>0$, for large enough $i$, we know that $\Sigma_{t_i}\cap (\rr^n\smallsetminus B_{r_0}(0))\subset (\Sigma_\infty)_\epsilon$ and that $|V_{t_i}(x)-V_\infty(x)|<\epsilon$ for all $|x|>r_0$ outside $(\Sigma_\infty)_\epsilon$.  Our goal is to show that $\Sigma_\infty$ is just the sphere of radius $\rsch(M)$, and then the result will follow from the Hausdorff convergence.

Suppose that part of $\Sigma_\infty$ lies inside the sphere of radius $\rsch(M)$.  Then we can find some $t_i$ and some point $x_0$ such that $x_0$ is outside $\Sigma^*(t_i)$ and yet $V_{t_i}(x_0)>0$, which is a contradiction.

We now come to the critical part of the proof.  Suppose that part of $\Sigma_\infty$ lies outside the sphere of radius $\rsch(M)$. Then for some $x_0\in\Sigma_\infty$ and some $r>0$, the ball $B_{2r}(x_0)$ lies completely outside the sphere of radius $\rsch(M)$.
The basic intuitive argument is as follows:  We know that $V_{t_i}$ is zero at $\Sigma^*(t_i)$, but in $B_r(x_0)$ we know that $V_\infty$ is significantly smaller than zero.  The only way this can happen is if the gradient of $V_{t_i}$ is blowing up.  In fact, we show that it blows up badly enough that the energy of $V_{t_i}$ blows up, which is a contradiction since we have a bound on energy (described below).

Consider the unique harmonic function $f$ that approaches $-1$ at infinity and is zero at the sphere $S_{\rmax}$.
Since $\Sigma^*(t_i)$ is contained in $S_{\rmax}$, we can deduce from the energy-minimizing property of harmonic functions that the energy of $V_{t_i}$ in the exterior of $\Sigma^*(t_i)$ is less than the energy of $f$ in the exterior of $S_{\rmax}$, namely $(n-2)\omega_{n-1}\rmax^{n-2}$.  Let $\Omega$ be the region outside $\Sigma^*(t_i)$, let $L_z=\{x\in B_r(x_0)\,|\, V_{t_i}(x)=z\}$, and let $dA_z$ be the area form of $L_z$. (Note that we suppress the dependence on $i$ in the notation.)  Then by the co-area formula and the H\"{o}lder inequality,
\begin{eqnarray*}
 (n-2)\omega_{n-1}\rmax^{n-2}
&\geq& \int_{\Omega} |\nabla V_{t_i}|^2\, dV\\
&\geq& \int_{\Omega\cap B_r(x_0)} |\nabla V_{t_i}|^2\, dV\\
&=&\int_{-1}^0 \left(\int_{L_z} |\nabla V_{t_i}|\,dA_z\right)\,dz\\
&\geq&\int_{-1}^0 |L_z|^2  \left(\int_{L_z} |\nabla V_{t_i}|^{-1}\,dA_z\right)^{-1}\,dz
\end{eqnarray*}
Let $\mu(z)=|\{x\in B_r(x_0)\,|\, V_{t_i}(x)>z\}|$.  Then $\mu'(z)=\int_{L_z} |\nabla V_{t_i}|^{-1}\,dA_z$ and we have
\begin{equation}\label{energy}
(n-2)\omega_{n-1}\rmax^{n-2}\geq \int_{-1}^0 |L_z|^2  \mu'(z) ^{-1}\,dz.
\end{equation}

On the other hand, we know that for some nonzero constant $c<0$, we have $V_\infty(x)<2c$ in $B_r(x_0)$.  Now let $\epsilon>0$, and choose $i$ large enough so that $V_{t_i}(x)<c<0$ for all $x\in B_r(x_0)$ lying outside $(\Sigma_\infty)_\epsilon$.  Therefore 
$$\{x\in B_r(x_0)\,|\, V_{t_i}(x)>c\}\subset (\Sigma_\infty)_\epsilon \cap B_r(x_0).$$
and it follows that
\begin{equation} \label{limitmu}
\lim_{i\to\infty} \mu(c) =0,
\end{equation} 
Since $\mu(c)=\int_c^0  \mu'(z)\,dz$, we can choose $i$ large enough so that $\mu'(z)<\sqrt{\mu(c)}$ on a set of measure at least $-c/2$.  We also know that for $0>z>c$, $L_z\subset (\Sigma_\infty)_\epsilon \cap B_r(x_0)$, and consequently these $L_z$'s are Hausdorff converging to $\Sigma_\infty\cap B_r(x_0)$.  In particular for $0>z>c$, $|L_z|$ is uniformly bounded below by some constant $\alpha$.  Plugging this into our energy bound (\ref{energy}), we see that
$$(n-2)\omega_{n-1}\rmax^{n-2}\geq \int_{-c}^0 |L_z|^2  \mu'(z) ^{-1}\,dz\geq (-c/2) \alpha^2 \mu(c)^{-1/2}$$
which contradicts equation (\ref{limitmu}).

\end{proof}
The main result follows easily from this lemma.  Let $\epsilon>0$.  Since $\Sigma^*(t)$ is outer minimizing with respect to $G_t$, we see that $A$ is less than or equal to the area of the sphere of radius $\rsch(M)+\epsilon$ with respect to $G_t$.  Also, the argument in the previous lemma shows that $U_t$ converges to $1+{M\over 2}|x|^{2-n}$ uniformly on $S_{\rsch(M)+\epsilon}$.  So for large enough $t$, we have
\begin{eqnarray*}
A&\leq& |S_{\rsch(M)+\epsilon}|_{G_t}\\
&=&\int_{S_{\rsch(M)+\epsilon}} U_t^{2(n-1)\over n-2}\,d\mu\\
&\leq&\int_{S_{\rsch(M)+\epsilon}} (1+{M\over 2}|x|^{2-n}+\epsilon)^{2(n-1)\over n-2}\,d\mu
\end{eqnarray*}
which converges to  $\omega_{n-1}(2M)^{n-1\over n-2}$ as $\epsilon\to 0$, proving our main theorem (Theorem \ref{main}) except for the case of equality.

Now consider the case of equality, $A=\omega_{n-1}(2m(0))^{n-1\over n-2}$.  Then $m(0)\leq M$, and thus $m(t)$ is constant. So we know that $\tilde{m}(t)=0$.  By the equality case of the Positive Mass Theorem, this means that $(\tilde{M}_t,\tilde{g}_t)$ is Euclidean space.\footnote{More precisely, because of the gluing we have to apply a refined version of the Positive Mass Theorem as in \cite{miao}.}   Therefore $(\bar{M}_t,\bar{g}_t)$ is globally conformal to Euclidean space and has nonnegative scalar curvature. This is only possible for the Schwarzchild metric.  Moreover, there is only one horizon in the Schwarzschild manifold, so that is where $\Sigma(t)$ must be.

The rest of the paper deals with the proofs that were skipped, namely Lemma \ref{Mnotzero} and Lemma \ref{rmax} (the primary technical lemma of this paper).

\section{Proof of Lemma \ref{Mnotzero}}

Suppose that $M=0$.  We want to argue that this is not possible.  We do this by following the same argument we gave in the $M\neq0$ case.  We can no longer choose $r_0<{1\over2}\rsch(M)$, but we can still choose some small $r_0>0$.  All of the arguments are valid up until we reach the proof of Lemma \ref{lastlemma}.  In the proof of this lemma, we argued that for some sequence of $t_i$'s, $\Sigma^*(t_i)\cap(\rr^n\smallsetminus B_{r_0})$ Hausdorff converges to some $\Sigma_\infty$.  A priori, $\Sigma_\infty$ could be empty.  However, the following lemma shows that, for a judicious choice of $r_0$ and $t_i$'s, $\Sigma_\infty$ will be nonempty.

\begin{lem}
Suppose that $M=0$.  Let $R$ be any constant such that $r_0<R<\left({A\over \omega_{n-1}}\right)^{1\over n-1}$.  There exists an unbounded sequence of times $t_i$ such that $\Sigma^*(t_i)$ always contains a point $x$ with $|x|>R$.
\end{lem}
A more general version of this fact is given in \cite[Section 9]{bray}.
\begin{proof}
Suppose, to the contrary, that there exists some $\tilde{t}>t_0$ such that for all $t>\tilde{t}$, $\Sigma^*(t)$ is contained in $S_R$, the sphere of radius $R$.  Choose $R'$ between $R$ and $\left({A\over \omega_{n-1}}\right)^{1\over n-1}$.
Since the $U_t$'s are harmonic and uniformly bounded, we know that some subsequence converges uniformly on compact subsets  of $\rr^n\smallsetminus \overline{B_R}$, and by Lemma \ref{converge} and the fact that $M=0$, we know that the limit function is $1$.
In particular, we see that $\lim_{t\to\infty} U_t(x)=1$ uniformly on $S_{R'}$.  Therefore, as $t\to\infty$, 
$|S_{R'}|_{G_t}=\int_{S_{R'}} U_t(x)^{2(n-1)\over n-2}\,d\mu\to\omega_{n-1}(C')^{n-1}<A$.  This contradicts the fact that
$|S_{R'}|_{G_t}\geq A$.
\end{proof}
So as long as we take $r_0<{1\over 2}\left({A\over \omega_{n-1}}\right)^{1\over n-1}$ and restrict our attention to times in the sequence described by the lemma, we know that $\Sigma_\infty$ is nonempty. Furthermore, for some $x_0\in\Sigma_\infty$ and some $r>0$, $B_{2r}(x_0)$ lies completely outside the sphere of radius $r_0$.  The energy argument given in Lemma \ref{lastlemma} now gives us a contradiction.

\section{Proof of Lemma \ref{rmax}}\label{rmaxproof}

Here we introduce a technical tool that will allow us to locally control the area of $\Sigma(t)$.  We 
will describe this tool in the language of integral currents, but this is not actually necessary for the 
application in this paper.

\begin{defin} 
For any $\gamma\geq 1$, an integral current $S$ in $\rr^n$ is said to be \emph{$\gamma$-almost area-minimizing} if, for any ball $B$ with $B\cap\spt \partial S=\emptyset$ and any integral current $T$ with $\partial T=\partial(S\lfloor B)$, $|S\lfloor B|\leq\gamma |T|$, where the absolute value signs denote the area.\footnote{In geometric measure theory, the correct term to use here is ``mass'' rather than ``area.''  We avoid the term ``mass'' here for obvious reasons.}
\end{defin}
Note that a $1$-almost area-minimizer is an area minimizer.  It is well-known that if $S$ is a $m$-dimensional minimal submanifold of $\rr^n$, then for any $x\in S$ and $0<r<d(x,\partial S)$, 
$$|S\cap B_r(x)|\geq {\alpha_m}r^m$$
where $\alpha_m=\omega_{m-1}/m$ is the volume of the unit ball in $\rr^m$.  This lower bound on area is a consequence of monotonicity \cite{allard}.  The following result is surely well-known to experts, but we include its proof for the sake of completeness.

\begin{lem}\label{gamma}
Let $\gamma\geq 1$, and let $S$ be a $m$-dimensional $\gamma$-almost area-minimizing integral current in 
$\rr^n$.  Let $x\in \spt S$, and let 
 $0<r<d(x,\spt\partial S)$.  Then
$$|S\lfloor B_r(x)|\geq \gamma^{1-m}\alpha_m r^m.$$
\end{lem}
\begin{proof}
Let $F(r)=|S\lfloor B_r(x)|$.  Since $F$ is monotonically increasing, $F'(r)$ exists for almost all $r$, 
and by the slicing theorem, 
$$F'(r)\geq |\partial(S\lfloor B_r(x))|.$$  
Let $G(r)$ be the infimum of all areas bounding 
$\partial(S\lfloor B_r(x))$. By the sharp\footnote{Using a non-sharp constant in the isoperimetric 
inequality would simply have the effect of attaining a worse constant in the statement of our proposition.} isoperimetric inequality \cite{almgren}, 
$$|\partial(S\lfloor B_r(x))|\geq m\alpha_m^{1/m} G(r)^{{m-1\over m}}.$$
Finally, by assumption we know that
$$G(r)\geq {1\over \gamma}F(r).$$
Putting the last three inequalities together, we find that
$$F'(r)\geq m\alpha_m^{1/m} \left({F(r)\over\gamma}\right)^{{m-1\over m}}.$$
Thus
$${d\over dr}\left(F(r)^{1\over m}\right)\geq \alpha_m^{1/m}\gamma^{{1-m\over m}}.$$
The result now follows from integrating this inequality.

\end{proof}

An equivalent formulation of Lemma \ref{rmax} is the following.
\begin{lem}\label{bounded}
There exists some $\rmax>0$ such that $\Sigma(t)$ is always enclosed by the coordinate sphere of radius $\rmax e^{2t\over n-2}$.
\end{lem}
\begin{proof}
First, choose $\rmax$ large enough so that $\Sigma(1)$ is contained in the coordinate sphere of radius $\rmax$.  Next, we will choose $\rmax$ large enough so that the following claim holds.
\begin{claim}
Let $R(t)=\rmax e^{2t\over n-2}$.  Choose any $T>0$.  Suppose that $\Sigma(t)$ is contained in the sphere of radius $R(t)$ for all $t\in[0,T]$.  Then $\Sigma(T+1)$ is contained in the sphere of radius $R(T+1)$.
\end{claim}
Clearly, if we can choose $\rmax$ large enough so that the claim is true, then we will have proved the lemma.  We know that for some constant $C$,
${1\over C}<\UU_0(x)<1+C|x|^{2-n}$ for all $|x|>\rharm$.  In all of the computations that follow, assume that $|x|>\rharm$.  That is, we are only interested in the ``harmonic part'' of the end.  Let $s<T$.  Since $\Sigma(s)$ is enclosed by the sphere of $R(s)$, the maximum principle tells us that for all $|x|>R(s)$, $v_s(x)$ is smaller than the unique $g_0$-harmonic function that is $0$ at the sphere of radius $R(s)$ and approaches $-e^{-s}$ at infinity.  Explicitly, 
$$v_s(x)\leq {1\over \UU_0(x)}e^{-s} \left(  \left({R(s)\over |x|}\right)^{n-2}-1\right).$$
Consequently,
\begin{eqnarray*}
\VV_s(x)&\leq& e^{-s} \left(  \left({R(s)\over |x|}\right)^{n-2}-1\right)\text{ for }|x|\geq R(s),\text{ while}\\
\VV_s(x)&\leq& 0\text{ for }|x|< R(s)
\end{eqnarray*}
Therefore, for $|x|\geq R(T)$
\begin{eqnarray*}
\UU_{T+1}(x)&\leq& 1+C|x|^{2-n} + \int_0^{T}  e^{-s} \left(  \left({R(s)\over |x|}\right)^{n-2}-1\right)\,ds\\
&=& 1 + C|x|^{2-n} + \left[e^{s} \left({\rmax\over |x|}\right)^{n-2} +e^{-s}\right]^{T}_0\\
&=&1 + C|x|^{2-n}+ \left[(e^{T}-1) \left({\rmax\over |x|}\right)^{n-2} +e^{-T}-1 \right]\\
&\leq& CR(T)^{2-n}+  \left[e^{T} \left({\rmax\over R(T)}\right)^{n-2} +e^{-T} \right]\\
&=& C e^{-2T}+2e^{-T}\\
&\leq& (2+C)e^{-T}
\end{eqnarray*}
On the other hand, since $v_s(x)\geq -e^{-s}$, we know that 
$$u_{T+1}(x)\geq 1+ \int_0^{T+1}-e^{-s}\,ds=e^{-(T+1)}.$$
Therefore $$\UU_{T+1}(x)\geq {1\over C}e^{-(T+1)}.$$

Now suppose that $\Sigma(T+1)$ contains a point $p$ with $|x|>R(T+1)$.  We will show that we can choose $\rmax$ so that the area of $|\Sigma(T+1)|_{g_{T+1}}$ is bigger than $A$, which is a contradiction.  Consider the coordinate ball $B$ of radius ${1\over 3}R(T)$ around $p$. For $n<8$, this ball $B$ lies outside the sphere of radius $R(T)$.  Using the bounds above, we see that for $x\in B$,
$$e^{-(T+1)}\leq \UU_{T+1}(x)\leq (2+C)e^{-T}.$$
From this we can conclude, straight from the definitions, that $\Sigma(T+1)$ is $\gamma$-almost area minimizing in $B$ with respect to the Euclidean metric, where $$\gamma= (2e+Ce)^{2(n-1)\over n-2}.$$
So by Lemma \ref{gamma}, $\Sigma(T+1)$ has Euclidean area greater than $\alpha_{n-1}\gamma^{2-n}(R(T)/3)^{n-1}$.  Therefore 
\begin{eqnarray*}
|\Sigma(T+1)|_{g_{T+1}}&\geq&(e^{-(T+1)})^{2(n-1)\over n-2} \alpha_{n-1}\gamma^{2-n}\left({R(T)\over 3}\right)^{n-1}\\
&=& e^{-{2(n-1)\over n-2}} \alpha_{n-1}\gamma^{2-n}\left({\rmax\over 3}\right)^{n-1},
\end{eqnarray*}
which is just some constant times $\rmax^{n-1}$.  We just need to chose $\rmax$ large enough so that this number is larger than $A$.
\end{proof}

\nocite{morgan}

\bibliographystyle{alpha}
\bibliography{research07}

\begin{thebibliography}{ADM61}

\bibitem[ADM61]{ADM}
R.~Arnowitt, S.~Deser, and C.~W. Misner.
\newblock Coordinate invariance and energy expressions in general relativity.
\newblock {\em Phys. Rev. (2)}, 122:997--1006, 1961.

\bibitem[All72]{allard}
William~K. Allard.
\newblock On the first variation of a varifold.
\newblock {\em Ann. of Math. (2)}, 95:417--491, 1972.

\bibitem[Alm86]{almgren}
F.~Almgren.
\newblock Optimal isoperimetric inequalities.
\newblock {\em Indiana Univ. Math. J.}, 35(3):451--547, 1986.

\bibitem[Bar86]{bartnik}
Robert Bartnik.
\newblock The mass of an asymptotically flat manifold.
\newblock {\em Comm. Pure Appl. Math.}, 39(5):661--693, 1986.

\bibitem[BI02]{brayiga}
Hubert~L. Bray and Kevin Iga.
\newblock Superharmonic functions in {$\bold R\sp n$} and the {P}enrose
  inequality in general relativity.
\newblock {\em Comm. Anal. Geom.}, 10(5):999--1016, 2002.

\bibitem[Bra01]{bray}
Hubert~L. Bray.
\newblock Proof of the {R}iemannian {P}enrose inequality using the positive
  mass theorem.
\newblock {\em J. Differential Geom.}, 59(2):177--267, 2001.

\bibitem[HI01]{huiilm}
Gerhard Huisken and Tom Ilmanen.
\newblock The inverse mean curvature flow and the {R}iemannian {P}enrose
  inequality.
\newblock {\em J. Differential Geom.}, 59(3):353--437, 2001.

\bibitem[Lee]{strongPMT}
Dan~A. Lee.
\newblock On the near-equality case of the {P}ositive {M}ass {T}heorem.
\newblock preprint.

\bibitem[Mia02]{miao}
Pengzi Miao.
\newblock Positive mass theorem on manifolds admitting corners along a
  hypersurface.
\newblock {\em Adv. Theor. Math. Phys.}, 6(6):1163--1182 (2003), 2002.

\bibitem[Mor95]{morgan}
Frank Morgan.
\newblock {\em Geometric measure theory}.
\newblock Academic Press Inc., San Diego, CA, second edition, 1995.
\newblock A beginner's guide.

\bibitem[Sch89]{montecatini}
Richard~M. Schoen.
\newblock Variational theory for the total scalar curvature functional for
  {R}iemannian metrics and related topics.
\newblock In {\em Topics in calculus of variations (Montecatini Terme, 1987)},
  volume 1365 of {\em Lecture Notes in Math.}, pages 120--154. Springer,
  Berlin, 1989.

\bibitem[SY79]{schyau}
Richard Schoen and Shing~Tung Yau.
\newblock On the proof of the positive mass conjecture in general relativity.
\newblock {\em Comm. Math. Phys.}, 65(1):45--76, 1979.

\bibitem[Wit81]{witten}
Edward Witten.
\newblock A new proof of the positive energy theorem.
\newblock {\em Comm. Math. Phys.}, 80(3):381--402, 1981.

\end{thebibliography}

\end{document}